\documentclass[reqno]{amsart}
\usepackage{txfonts}
\usepackage{amsfonts}
\usepackage{amssymb}
\usepackage[mathscr]{eucal}
\usepackage{amsmath}
\usepackage[bookmarksopen,colorlinks,citecolor=blue,
linkcolor=black,pdfstartview=FitH]{hyperref}
\usepackage{ytableau}
\usepackage{float}

\newtheorem{theorem}{Theorem}[section]
\newtheorem{corollary}[theorem]{Corollary}
\newtheorem{lemma}[theorem]{Lemma}
\newtheorem{proposition}[theorem]{Proposition}
\theoremstyle{definition}
\newtheorem{definition}[theorem]{Definition}

\newtheorem{problem}[theorem]{Problem}

\newtheorem*{conjecture}{Conjecture}
\numberwithin{equation}{section}
\theoremstyle{remark}
\newtheorem{remark}[theorem]{Remark}

\allowdisplaybreaks[4]

\begin{document}

\title{A virtual character and nonzero Kronecker coefficients for self-conjugate partitions}
\pdfbookmark[0]{A virtual character and nonzero Kronecker coefficients for self-conjugate partitions}{}


\author{Xin Li}
\address{Department of Applied Mathematics, Zhejiang University of Technology, Hangzhou 310023, P. R. China}
\email{xinli1019@126.com}

\thanks{The research is supported by National Natural Science Foundation of China (Grant No.11626211).}


\begin{abstract}
We generalize  Regev's result on a virtual character of $S_n$.
Suppose that $\lambda$ and $\mu$ are integer partitions of $n$. For the associated irreducible character $\chi^\lambda$ of $S_n$,  when $\chi^\lambda(\mu)\neq0$ we find another partition $\tau$ related to $\lambda$ such that $\chi^\tau(\mu)\neq0$ by the virtual character. Applying this result, we obtain a class of nonzero Kronecker coefficients by Pak et al.'s character criterion. Moreover, we discuss the effectiveness of Pak et al.'s character criterion by a concrete example.
\end{abstract}

\maketitle

\section{Introduction}\label{se:intro}
The Kronecker product problem is a problem of computing multiplicities in  the internal tensor product of two irreducible symmetric group representations. The multiplicity can be characterized by the corresponding characters:
$$g(\lambda,\mu,\nu) = \langle\chi^\lambda,\chi^\mu\otimes\chi^\nu\rangle.$$
It is often referred as `classic' and `one of the last major open problems' in algebraic combinatorics \cite{PPV}.  There are several related problems such as give an combinatoric explanation of  $g(\lambda,\mu,\nu)$ and determine when $g(\lambda,\mu,\nu)>0$.  The study of Kronecker coefficients have connection with other research fields such as quantum information theory \cite{Christ06,Christ07,Klyachko} and geometric complexity theory \cite{BlasiakMS}.

In \cite{Heide}, Heide, Saxl, Tiep and Zalesski proved that with a few known exceptions, every irreducible character of a simple group of Lie type is a constituent of the tensor square of the Steinberg character. In 2012, J. Saxl made the following conjecture \cite{PPV}:
\begin{conjecture}
Denote by $\rho_k = (k, k-1,... , 2, 1) \vdash n$, where $n ={k+1 \choose 2}$. Then for every $k\geq 1$, the tensor square $\chi^{\rho_k} \otimes \chi^{\rho_k}$ contains every irreducible character of $S_n$ as a constituent.
\end{conjecture}
Motivated by Saxl's conjecture, recently many results have been obtained \cite{Bessenrodt17,Ikenm,Luo,PPV}. In \cite{PPV}
Pak et al. showed that if $\mu=\mu'$ is self-conjugate and $\chi^\lambda(\hat{\mu})\neq0$, then $g(\lambda,\mu,\mu)>0$.
This criterion provides a connection between nonzero Kronecker coefficients and character values of the symmetric group.
Recently, Bessenrodt provided a very different approach and showed that results on the spin characters of a double cover $\tilde{S}_n$ of the symmetric group $S_n$ can be fruitfully applied towards Saxl's conjecture \cite{Bessenrodt17}. All these results illustrate that character theory is a powerful tool for understanding Kronecker coefficients.

Inspired by Pak et al.'s character criterion, we want to extend their discussions on nonzero irreducible character value and ask if $\chi^\lambda(\mu)\neq0$ implies another $\tau$ in some neighborhood (see Defination \ref{def:neib}) of $\lambda$ such that $\chi^\tau(\mu)\neq0$.  So if such $\tau$ exists we will get more nonzero Kronecker coefficients by the Main Lemma of \cite{PPV}.  A recursion formula to calculate the character of $S_n$ is the well-known Murnaghan-Nakayama rule. It can be rephrased by the so called `wrap operator' \cite[Thm. 21.7]{James78}. An integer-combination of irreducible characters is called a virtual character. In \cite{Regev}, Regev discussed a virtual character which is actually constructed by the `wrap operator' under some constraints. The construction there was called the `going around process'.  Regev gave an explicit expression of the virtual character's value.  Recently, Morotti estimated the the number of non-zero character values in generalized blocks of symmetric groups \cite{Morotti}. The proof there relied on a virtual character constructed by the `wrap operator'. In Theorem \ref{thm:ssum}, we generalize Regev's result on the virtual character constructed by the `going around process'. Through this and Morotti's method, we show that if $\chi^\lambda(\mu)\neq0$ then in the neighborhood constructed by the `going around process' there exists a class of different partitions whose character values on $\mu$ are nonzero. Implied by this and the properties of $\beta$-set \cite{Olsson}, we find partitions in the neighborhood of $\rho_k$ and obtain a class of nonzero Kronecker coefficients by Pak et al.'s criterion.

By the Frobenius's characteristic map \cite{sagan}, there is a bijection between irreducible characters and Schur functions.  Kronecker coefficients can also be given by the expression
$$s_{\mu}*s_{\nu}=\sum_{\lambda}g(\lambda,\mu,\nu)s_{\lambda},$$
where the Kronecker product of two Schur functions is decomposed through the combination of Schur basis \cite{rosas}. In Section 8 of \cite{PPV}, the authors discussed the effectiveness of their character criterion, that is, how many nonzero Kronecker coefficients can be detected by their main lemma.
In the second part of this paper, we continue their discussion for another self-conjugate partition. By a formula of Littlewood \cite{Lwood}, we find the nonzero coefficients in the decomposition of $s_{\gamma}*s_{\gamma}$ where $\gamma=(k+1,2,1^{k-1})$ ($k\in \mathbb{N}$).
Through this example we can see that the character criterion is less effectiveness for $s_{\gamma}*s_{\gamma}$. It also provides a new example of Kronecker coefficients in the condition of Durfee size 2.

The organization of the paper is as follows. In Section 2, we summarize basic definitions and results needed in this paper. We generalize Regev's result on a virtual character in Section 3. We show that if $\chi^\lambda(\mu)\neq0$ then there exists a class of $\tau$ in the `neighborhood' of  $\lambda$ such that $\chi^\tau(\mu)\neq0$. Implied by this result, we discuss how to find $\tau$ when $\lambda$ is of staircase shape and obtain a class of new nonzero Kronecker coefficients by Pak et al.'s criterion. In Section 4, we discuss the effectiveness of Pak et al.'s character criterion by an example.

\section{Preliminaries}\label{se:preli}
We let $S_n$ denote the symmetric group on $n$ letters. We assume the readers are familiar with the basic notations and results of representation theory of symmetric groups and related combinatorics. They can be found in \cite{James78,JKer81,sagan,stanley}.

If $A$ is a set, the cardinality of $A$ is denoted by $|A|$. A partition $\lambda$ of $n$, denoted $\lambda\vdash n$, is defined to be a weakly decreasing sequence $\lambda= (\lambda_1,\lambda_2,\ldots,\lambda_k)$ of non-negative integers such that the sum $\lambda_1+\lambda_2+\cdots+\lambda_k=n$. The set of all partitions of $n$ is denoted by $P(n)$. We do not distinguish between a partition $\lambda$ and its Young diagram. The Young diagram of a partition $\lambda$ is thought of as a collection of
boxes arranged using matrix coordinates. Denote by $\lambda'$ the conjugate partition of $\lambda$. Partition $\lambda$ is called self-conjugate if $\lambda=\lambda'$.
The Durfee size $d(\lambda)$ of a partition $\lambda  $ is $d$ if ($d^d$) is the largest square contained in its Young diagram.  For $P(n)$, the subset of partitions with Durfee size $k$ is denoted by $DS(k,n)=\{\lambda\in P(n)| d(\lambda)=k\}$.
For a box with position $(i,j)$ in the Young diagram of $\lambda$, denote the corresponding hook by $h_{i,j}$ and the hooklength by $|h_{i,j}|$.
For $\lambda\vdash n$ with $d(\lambda)=s$,  the principal hook partition of $\lambda$ is defined by $\hat{\lambda}=(|h_{1,1}|, . . . , |h_{s,s}|)$  where $h_{1,1}, . . . , h_{s,s}$ are the principal hooks in $\lambda$. Observe that $\hat{\lambda}\vdash n$.

The {\it hooklength diagram} of $\lambda$ is obtained by putting each box with the corresponding hooklength and denoted by $H(\lambda)$. For example, the hooklength diagram of $\lambda=(4,3,2,1)$ is:

\centerline{
\begin{ytableau}
7& 5 &  3 & 1\\
5& 3 & 1 \\
3& 1\\
1
\end{ytableau}}
\noindent where $|h_{1,1}|=7$,  $|h_{2,2}|=3$.

For partitions $\lambda$, $\mu$ with $|\lambda| = |\mu| = n$, let $\chi^\lambda$ be the associated irreducible character of $S_n$ and $\chi^\lambda_{\mu}$ or  $\chi^\lambda(\mu)$ the value $\chi^\lambda$ takes on the associated conjugacy class. For each $h\in H(\lambda)$, there corresponds a rim hook on the the boundary
of $\lambda$  by projecting it along diagonals. We let $\lambda\setminus h$ be the partition obtained by removing the rim hook corresponding to $h$.  Denote the leg length of $h$ by $\ell(h)$. Suppose $\lambda$, $\mu\in P(n)$ with $\mu=(\mu_1,\mu_2,\ldots,\mu_k)$. By the bijection between rim hooks and regular hooks, the Murnaghan-Nakayama rule (or `the M-N rule') \cite[Thm. 4.10.2]{sagan} says that
  \begin{equation*}
  \chi^\lambda_\mu=\sum_{h}(-1)^{\ell(h)}\chi^{\lambda\setminus h}_{\mu\setminus\mu_1},
  \end{equation*}
where the sum runs over all $h\in H(\lambda)$ with length $\mu_1$. If $\mu_1 \notin H(\lambda)$, we have $ \chi^\lambda_\mu=0$.

\section{The neighborhood of nonzero character value and Kronecker coefficient}\label{se:onnozero}
Let $n\geq r$, $\rho\vdash n-r$ and $\mu\vdash n$. Assume
$\rho\subseteq\mu$ and $S$ is a rim hook of $\mu$ with $\mu \setminus S = \rho$, then we write $\mu= \rho*S$. Let $\ell(S)$ denote the leg length of $S$.
For each $\rho\vdash n-r$, in the following we let
\begin{equation}\label{eq:psirho}
\psi_{\rho,n}=\sum_{S\in B(\rho,r)}(-1)^{\ell(S)}\chi^{\rho*S},
\end{equation}
where  $B(\rho,r)\subseteq P(n)$  is defined by
$$B(\rho,r)=\{\rho*S\mid~S~\text{are all possible rim hooks of length $r$ such that}~\rho*S\in P(n)\}.$$
So $\psi_{\rho,n}$ is a virtual character of $S_n$. Properties of $\psi_{\rho,n}$ have been studied in \cite[Thm. 3.1]{Regev} and \cite[Thm. 21.7]{James78}. In \cite{James78}, for $\mu\in P(n)$ James provided the condition when $\psi_{\rho,n}(\mu)=0$. However, the proof there is implicit. In \cite{Regev}, Regev discussed  the relation between $\psi_{\rho,n}$ and $\chi^\rho$ under the condition $n\geq 2(n-r)+2$.
In this section,  we generalize Regev's method and discuss the relation between $\psi_{\rho,n}$ and $\chi^\rho$ for general $n$ and $r$ (see Theorem \ref{thm:ssum} below). We will determine when $\chi^\lambda(\mu)\neq0$  implies another in the neighborhood (see Definition \ref{def:neib}) of $\lambda$. Combining with Pak's character criterion in \cite{PPV} we will find more nonzero Kronecker coefficients.

In order to illustrate Theorem \ref{thm:ssum}, the following notations for partitions will be used as convenient. Suppose that $\mu=(\mu_1,\mu_2,\ldots,\mu_l) \in P(n)$. Recall that there is a one-to-one correspondence between partitions in $P(n)$ and conjugacy classes of $S_n$, for example \cite{sagan}. So we can rewritten $\mu$ as $\mu=c_1c_2\cdots c_l$ where $c_i$ ($i=1,2,\ldots,l$) are cycles with length $\mu_i$.  In the following, for some cycle $c_{i}$ we denote $\mu=c_{i}\bar{\mu}_i$ where $\bar{\mu}_i\in P(n-\mu_i)$ is the partition corresponding to $c_1\cdots c_{i-1}c_{i+1}\cdots c_{l}$. The multiplicity of $r$-cycles in $\mu$ is the total number of cycles with length $r$ in $\{c_1,c_2,\cdots,c_l\}$ and denoted by $m_{r,\mu}$. The {\it cycle type} of $\mu$  is denoted by $C(\mu)=(1^{m_{1,\mu}},2^{m_{2,\mu}},\ldots,n^{m_{n,\mu}})$.  Without of confusion, we don't distinguish between $\mu$ and the corresponding conjugacy class.

Based on these notations, we have the following main theorem of this section which generalizes Theorem 3.1 in \cite{Regev}. It also gives a complement of Theorem 21.7 in \cite{James78}, which tells us the nonzero value of the generalized character there.

\begin{theorem}\label{thm:ssum}
For each $\mu\in P(n)$ and $\psi_{\rho,n}$ defined in (\ref{eq:psirho}), if $\mu=c\bar{\mu}$ where $c$ is a cycle of length $r$ and $\bar{\mu}\in P(n-r)$, then we have $\psi_{\rho,n}(\mu)=rm_{r,\mu}\cdot \chi^{\rho}(\bar{\mu})$ where $m_{r,\mu}$ is the multiplicity of $r$-cycles in $\mu$. Specially, if the cycle type of  $\mu$ contains no $r$-cycle (i.e. $m_{r,\mu}=0$), then $\psi_{\rho,n}(\mu)=0$.
\end{theorem}

\begin{remark}
Suppose that $n\geq 2k+2$, $r=\mu_1=n-k$. Then in Theorem \ref{thm:ssum} above we have $m_{n-k,\mu}=1$ and we get Theorem 3.1 of \cite{Regev}.
\end{remark}

The proof is given below. In the following we assume $|P(n-r)|=d$. Then after putting an order on $P(n-r)$ we denote $P(n-r)=\{\rho_1,\rho_2,\ldots,\rho_d\}$.

\begin{proposition}\label{prp:suml}
Suppose $\nu\vdash n-r$ and $c$ is a cycle of length $r$. Then
$$\sum_{\lambda\vdash n}\chi^{\lambda}_{c\nu}\chi^{\lambda}=
\sum_{\rho\vdash n-r}\chi^{\rho}_{\nu}\psi_{\rho,n}.$$
\end{proposition}
\begin{proof}
By the M-N rule we have,

$\begin{aligned}\label{eq-slam}
\chi^{\lambda}_{c\nu}\chi^{\lambda}
&=(a_1^{\lambda}\chi^{\rho_1}_{\nu}
+a_2^{\lambda}\chi^{\rho_2}_{\nu}+\cdots+a_d^{\lambda}\chi^{\rho_d}_{\nu})
\chi^{\lambda} \\
&=a_1^{\lambda}\chi^{\rho_1}_{\nu}\chi^{\lambda}
+a_2^{\lambda}\chi^{\rho_2}_{\nu}\chi^{\lambda}+\cdots+
a_d^{\lambda}\chi^{\rho_d}_{\nu}\chi^{\lambda}\\
&=a_1^{\lambda}\chi^{\rho_1}_{\nu}\chi^{\rho_1*S_1^{\lambda}}
+a_2^{\lambda}\chi^{\rho_2}_{\nu}\chi^{\rho_2*S_2^{\lambda}}+\cdots+
a_d^{\lambda}\chi^{\rho_d}_{\nu}\chi^{\rho_d*S_d^{\lambda}},
\end{aligned}$\\
where $a_i^{\lambda}=(-1)^{\ell(S_i^{\lambda})}$ if there exists a rim hook $S_i^{\lambda}$  of length $r$ such that $\lambda=\rho_i*S_i^{\lambda}$, otherwise,
 $a_i^{\lambda}=0$ for $i=1,2,\ldots,d$.

So if we take sum for all $\lambda\in P(n)$ over both side of equations above, we get the desired result.
\end{proof}

\begin{lemma}(The second orthogonality relation \cite{Isaacs})\label{le:tco}
\begin{equation*}
\sum_{\lambda\vdash n}\chi^{\lambda}_{\mu}\chi^{\lambda}_{\nu}=\left\{ \begin{aligned}
         |Z_{S_n}(\mu)|&\qquad\text{if}\quad\mu=\nu\\
         0 &\qquad\text{if}\quad\mu\neq\nu
                          \end{aligned} \right.,
                          \end{equation*}
where $\lambda,\mu,\nu\in P(n)$ and $Z_{S_n}(\mu)$ is the centralizer of $\mu$ in $S_n$.
\end{lemma}

For the number of $Z_{S_n}(\mu)$, we have:
\begin{lemma}\label{lem:cenn}
Suppose $\mu \vdash n$ with cycle type $C(\mu)=(1^{m_{1,\mu}},2^{m_{2,\mu}},\ldots,n^{m_{n,\mu}})$. Then
$$|Z_{S_n}(\mu)|=\prod_{i=1}^{n}i^{m_{i,\mu}}m_{i,\mu}!.$$
\end{lemma}

\begin{corollary}\label{cor:cenv}
Suppose that a conjugacy class in $S_n$ can be written as $c\nu$ where $c$ is a cycle of length $r$. Then $\nu\in P(n-r)$ and we have
\begin{equation*}
\sum_{\lambda\vdash n}\chi^{\lambda}_{\mu}\chi^{\lambda}_{c\nu}=\left\{ \begin{aligned}
         |Z_{S_{n-r}}(\nu)|\cdot rm_{r,\mu}&\qquad\text{if}\quad\mu=c\nu\\
         0 &\qquad\text{if}\quad\mu\neq c\nu
                          \end{aligned} \right.,
                          \end{equation*}
where $\mu\in P(n)$.
\end{corollary}
\begin{proof}
If $\mu=c\nu$ for a cycle $c$ of length $r$ with cycle type $(1^{m_{1,\mu}},2^{m_{2,\mu}},\ldots,n^{m_{n,\mu}})$, then the cycle type of $\nu$ is $(1^{m_{1,\mu}},\ldots,r^{m_{r,\mu}-1},\ldots,n^{m_{n,\mu}})$. Hence, the proof is completed by Lemma \ref{le:tco} and \ref{lem:cenn}.
\end{proof}

For $\rho_i\in P(n-r)$ ($i=1,2,\ldots,d$), let $K=[\chi_{\rho_i}^{\rho_j}]_{i\times j}$ be the $d\times d$ matrix whose transpose can be viewed as the character table of $S_{n-r}$. Define a $d\times1$ column vector by $[\psi_{\rho_i,n}]$. Then by Proposition \ref{prp:suml} and Corollary \ref{cor:cenv}  we have,
\begin{equation}\label{eq:nozc}
[\chi_{\rho_i}^{\rho_j}][\psi_{\rho_i,n}(\mu)]=[\sum_{\lambda\vdash n}\chi^{\lambda}_{c\rho_i}\chi_{\mu}^{\lambda}]\\
=\left(
         \begin{array}{c}
            0 \\
            \vdots \\
            |Z_{S_{n-r}}(\rho_{i_0})|\\
            \vdots\\
            0 \\
          \end{array}
        \right)\cdot rm_{r,\mu},
\end{equation}
if $\mu=c\rho_{i_0}$ for some $i_0=1,2,\ldots,d$. Otherwise,
\begin{equation}\label{eq:zoc}
[\chi_{\rho_i}^{\rho_j}][\psi_{\rho_i,n}(\mu)]=\mathbf{0}.
\end{equation}

Next, we give the proof of Theorem \ref{thm:ssum}.
\begin{proof}
It is well known that $K=[\chi_{\rho_i}^{\rho_j}]$ is invertible \cite[Cor~6.5]{James78}. If  $\mu\neq c\rho_{i_0}$ for some $i_0=1,2,\ldots,d$, then by (\ref{eq:zoc}) we have
$$[\psi_{\rho_i,n}(\mu)]=\mathbf{0}.$$
Hence, we get the second part of the theorem.

In the following, we let $\mu_j=c \rho_j$ ($j=1,2,\ldots,d$) and the $d\times d$ matrix defined by
$$M=[\psi_{\rho_i,n}(\mu_j)]_{i\times j}=[\psi_{\rho_i,n}(c\rho_j)].$$
Then by (\ref{eq:nozc}) we have
\begin{equation}\label{eq:kmkd}
\begin{aligned}
KM=&[\chi_{\rho_i}^{\rho_j}][\psi_{\rho_i,n}(\mu_j)]=[\sum_{\lambda\vdash n}\chi^{\lambda}_{c\rho_i}\chi_{\mu_j}^{\lambda}]\\
=&diag \left(|Z_{S_{n-r}}(\rho_{1})|,\ldots,|Z_{S_{n-r}}(\rho_{d})|\right)\cdot
diag \left(rm_{r,\mu_1},rm_{r,\mu_2},\ldots, rm_{r,\mu_d} \right)\\
=&KK^{T}D_{r},
\end{aligned}
\end{equation}
where $D_{r}=diag \left(rm_{r,\mu_1},rm_{r,\mu_2},\ldots, rm_{r,\mu_d} \right)$.

Since $K$ is invertible, by (\ref{eq:kmkd}) we have $M=K^{T}D_r$. That is,
$$[\psi_{\rho_i,n}(\mu_j)]=[\chi_{\rho_j}^{\rho_i}]\cdot diag\left(rm_{r,\mu_1},rm_{r,\mu_2},\ldots, rm_{r,\mu_d} \right).$$
Comparing elements of matrix in both side, we have $\psi_{\rho_i,n}(\mu_j)= rm_{r,\mu_j}\cdot \chi_{\rho_j}^{\rho_i}$ for $i,j=1,2,\ldots,d$, which completes the first part of the theorem.
\end{proof}

In the following, we will use Theorem \ref{thm:ssum} to find more nonzero irreducible character values in the neighborhood (see Definition \ref{def:neib}) of a partition. Then by Pak's character criterion we will find more nonzero Kronecker coefficients.

\begin{lemma}\cite{Morotti}\label{lem:sgn}
Assume that $\gamma_1$, $\gamma_2$, $\delta_1$ and $\delta_2$ are partitions such that $\gamma_i$ can be obtained from $\delta_j$  by removing a hook of leg length $l_{i,j}$ for $1 \leq i, j \leq 2$. Further assume that $\gamma_1\neq\gamma_2$, $\delta_1\neq\delta_2$ and $\delta_1$ and $\delta_2$ cannot be obtained one from the other by removing a hook. Then $(-1)^{l_{1,1}+l_{1,2}+l_{2,1}+l_{2,2}}=-1$.
\end{lemma}

\begin{definition}\label{def:neib}
 Suppose $\mu\in P(n)$ and $h\in H(\mu)$ with $|h|=r$. Let $\bar{\mu}_h=\mu \backslash h$.  Define the $|h|$-neighborhood of $\mu$ by
$$N(\mu,|h|)=\{\bar{\mu}_h*S\mid~S~\text{are all possible rim hooks of length $|h|=r$ such that}~\bar{\mu}_h*S\in P(n)\}.$$
\end{definition}

By method in \cite[Thm.3]{Morotti} and Theorem \ref{thm:ssum}, we have the following theorem:
\begin{theorem}\label{thm:neib}
If $\chi^{\mu}(\lambda)\neq0$ and $\lambda=(\lambda_1,\lambda_2,\ldots, \lambda_k)$, then for each $h\in H(\mu)$ with $|h|=r$ and $|h|\notin \{\lambda_1,\lambda_2,\ldots, \lambda_k\}$, there exists a $\tau\in N(\mu,|h|)$ such that $\chi^{\tau}(\lambda)\neq0$ and $\tau\neq\mu$. Moreover, if $h_1, h_2,\ldots,h_l\in H(\mu)$ with $|h_i|\notin \{\lambda_1,\lambda_2,\ldots, \lambda_k\}$ are pairwise different, then the corresponding $\tau_i$ ($i=1,2,\ldots,l$) are also different.
\end{theorem}
\begin{proof}
Suppose that $|h|=r$. Define a virtual character of $S_n$ by
$$\psi_{\mu,r}=\sum_{\bar{\mu}_h*S\in N(\mu,|h|)}(-1)^{\ell(S)}\chi^{\bar{\mu}_h*S}.$$
Apparently, $\psi_{\mu,r}$ is just the virtual character discussed in Theorem \ref{thm:ssum}. By assumption $\lambda$ contains no cycles of length $r$, so we have
$\psi_{\mu,r}(\lambda)=0.$ By definition we know that $\mu\in N(\mu,|h|)$. Assume that
$\mu=\bar{\mu}_h*S_{\mu}$, so there exists another $\tau\in N(\mu,|h|)$ such that
$\chi^{\tau}(\lambda)\neq0$ and $\tau\neq\mu$. Let $\tau=\bar{\mu}_h*S_{\tau}$ for some rim hook $S_{\tau}$. Then we know that $(-1)^{\ell(S_\tau)}\chi^{\tau}(\lambda)$ and $(-1)^{\ell(S_\mu)}\chi^{\mu}(\lambda)$ have different signs.

Suppose that $h_i\in H(\mu)$ for $1\leq i \leq l$ which are different and $|h_i|\notin\{\lambda_1,\lambda_2,\ldots, \lambda_k\}$. Then by previous discussion there exist $\tau_i$ ($1\leq i \leq l$)  which satisfy $\chi^{\tau_i}(\lambda)\neq0$. Next, we show that if $h_i\neq h_j$, then $\tau_i\neq \tau_j$. Without loss of generality, we assume that $i=1$, $j=2$.
Conversely, suppose that $\tau_1=\tau_2=\tau$. Let $\gamma_1=\bar{\mu}_{h_1}$, $\gamma_2=\bar{\mu}_{h_2}$ , $\delta_1=\mu$ and $\delta_2=\tau$. There exist $S_{\mu}^i$, $S_{\tau}^i$ ($i=1,2$) such that $\mu=\bar{\mu}_{h_1}*S_{\mu}^1=\bar{\mu}_{h_2}*S_{\mu}^2$ and $\tau= \bar{\mu}_{h_1}*S_{\tau}^1=\bar{\mu}_{h_2}*S_{\tau}^2$.
Then by Lemma \ref{lem:sgn} we have $(-1)^{\ell(S_{\mu}^1)}(-1)^{\ell(S_{\mu}^2)}
(-1)^{\ell(S_{\tau}^1)}(-1)^{\ell(S_{\tau}^2)}=-1$. So three signs of the leg lengths are same except the other one. Suppose the sign of $(-1)^{\ell(S_{\tau}^2)}$ is different with others, especially with $(-1)^{\ell(S_{\mu}^2)}$. If $\chi^{\mu}(\lambda)$, $\chi^{\tau}(\lambda)\neq0$ have different signs, then $(-1)^{\ell(S_{\tau}^2)}\chi^{\tau}(\lambda)$ and $(-1)^{\ell(S_{\mu}^2)}\chi^{\mu}(\lambda)$ have the same sign which contradicts with previous discussion. If $\chi^{\mu}(\lambda)$, $\chi^{\tau}(\lambda)\neq0$ have the same sign, similar discussion can be used to get the contradiction.
\end{proof}

In \cite{PPV} Pak et al gave a method to determine the positivity of Kronecker coefficients by characters. In the following, we use Theorem \ref{thm:neib} to find more nonzero Kronecker coefficients. Combining Theorem \ref{thm:neib} and Lemma 1.3 in \cite{PPV} we have the following corollary.
\begin{corollary}\label{cor:nbkc}
Let $\mu=\mu^\prime$ be a self-conjugate partition of $n$, and let $\hat{\mu}=
(2 \mu_1 - 1, 2 \mu_2 - 3, 2 \mu_3 - 5,\ldots)\vdash n$ be the principal hook partition. Suppose $\chi^\lambda(\hat{\mu})\neq0$ for some $\lambda\vdash n$. For each hook $h\in H(\lambda)$ with length $|h|=r$, if $r\neq 2 \mu_1 - 1, 2 \mu_2 - 3, 2 \mu_3 - 5,\ldots$ the hooklengths of principal hooks of $\mu$, then there exists a $\tau\in N(\lambda,|h|)$ with $\tau\neq\lambda$ such that $\chi^{\tau}(\hat{\mu})\neq0$ and
$\chi^{\tau}\in \chi^{\mu}\otimes\chi^\mu$.
\end{corollary}

\begin{corollary}\label{cor:dshl}
 For partitions $\mu\vdash n$, $\hat{\mu}=(\hat{\mu}_1,\hat{\mu}_2,\ldots,\hat{\mu}_k)$ where $k=d(\mu)$.
Let $A=\{\lambda|d(\lambda)\leq2, \chi^\lambda(\hat{\mu})\neq0\}$,
 $B=\{\nu|k-1\leq d(\nu)\leq k+1, \chi^\nu(\hat{\mu})\neq0\}$. Let $h(\mu,\hat{\mu})$ be the number of boxes in the Young diagram of $\mu$ with no hooklengths in $\{\hat{\mu}_1,\hat{\mu}_2,\ldots,\hat{\mu}_k\}$. Then we have
\begin{enumerate}
\item $|A|\geq n-k+1$,
\item $|B|\geq h(\mu,\hat{\mu})+1$.
\end{enumerate}
\end{corollary}
\begin{proof}
(1) For $(n)\in P(n)$, we have $\chi^{(n)}(\hat{\mu})=1$ and the Young diagram of $(n)$ consists of boxes with hooklengths $\{1,2,\ldots,n\}$ exactly once. Hence, there are $n-k$ boxes in the Young diagram of $(n)$ with hooklengths $r\notin \{\hat{\mu}_1,\hat{\mu}_2,\ldots,\hat{\mu}_k\}$.
By Theorem \ref{thm:neib}, besides $(n)$ there exists $n-k$ distinct partitions whose character on $\hat{\mu}$ are nonzero. For each $h\in H((n))$, by definition we know that the Durfee size of partitions in $N((n),|h|)$ is not greater than $2$.

(2) The proof is similar to (1), if we use $\chi^{\mu}(\hat{\mu})=(-1)^{n-d(\mu)/2}$ to estimate $|B|$.
\end{proof}

In Corollary \ref{cor:dshl}, we can see that if $k\geq4$, then elements in $A$ and $B$ are different. For $\lambda$, $\tau\in P(n)$, if $\hat{\tau}=\hat{\mu}$ we have $\chi^{\tau}(\hat{\mu})=1$ or $-1$. So in \cite{PPV}, the authors gave an estimation of nozero characters on the shape of staircase and caret \cite[Prop. 4.14]{PPV}. However, there are only one partitions with the same principal hook partition as the chopped square. So we can use Corollary \ref{cor:dshl} to obtain a better lower bound for the  chopped square partition. The following is the hooklength diagram of $\eta_5=(5,5,5,5,4)$:

\begin{center}
\begin{ytableau}
9& 8 & 7 & 6&4\\
8& 7 & 6 & 5&3\\
7& 6 & 5 & 4&2\\
6& 5 & 4 & 3&1\\
4& 3 & 2 &1
\end{ytableau}

\end{center}

\begin{corollary}
Suppose that $\eta_k=(k^{k-1},k-1)\vdash k^2-1$ and $n=k^2-1$. Let $\Psi(\eta_k)=\{\lambda\in P(n)| \chi^\lambda(\hat{\eta}_k)\neq0\}$. Then we have
\begin{enumerate}
  \item $|\Psi(\eta_k)|\geq \frac{3}{2}k^2-k+1$, if $k\geq4$ is even;
  \item $|\Psi(\eta_k)|\geq \frac{3k^2+5}{2}-k$, if $k\geq4$ is odd.
\end{enumerate}
\end{corollary}
\begin{proof}
By the notation in Corollary \ref{cor:dshl}, we have that
$$|\Psi(\eta_k)|\geq |A|+|B|=n-k+h(\eta_k,\hat{\eta}_k)+2,$$
for $k\geq4$. The number $h(\eta_k,\hat{\eta}_k)$ can be obtained by firstly counting the boxes besides the last column and row of $H(\eta_k)$ and then others. Then the result is followed by
\begin{equation*}
\begin{split}
h(\eta_k,\hat{\eta}_k)
=\left\{ \begin{aligned}
&\frac{k^2}{2}              \quad\text{if}~k ~\text{is even}\\
& \frac{k^2+3}{2}           \quad\text{if}~k~\text{is odd}
        \end{aligned} \right..
\end{split}
\end{equation*}
\end{proof}

\subsection*{Application: a class of nonzero Kronecker coefficients related to staircase shape}

\

Corollary \ref{cor:nbkc} just implies the existence of nonzero irreducible characters. By the properties of $\beta$-set, in the following we will find them by an example: partitions with staircase shape. And then we will obtain a class of nonzero Kronecker coefficients that can be detected by Pak et al.'s criterion.

Let $\rho_m=(m,m-1,\ldots,1)\vdash n$ ($m\in \mathbb{N}^+$) be the partitions with staircase shape. Take $m=2k$ for example. The hooklengths of $\rho_m$ are consisted of $\{1,3,5,\ldots,4k-1\}$ and $\hat{\rho}_m=(4k-1,4k-5,\ldots,7,3)$.
For those $h\in H(\rho_m)$ with length $|h|\notin \{4k-1,4k-5,\ldots,7,3\}$, there exists $\tau\in N(\rho_m, |h|)$ such that $\chi^\tau(\hat{\rho}_m)\neq0$ and $\tau\neq \rho_m$ by the fact $\chi^{\rho_m}(\hat{\rho}_m)\neq0$ and Corollary \ref{cor:nbkc}. We will find such $\tau$ in $N(\rho_m, |h_{1,j}|)$, $N(\rho_m, |h_{j,1}|)$ and $N(\rho_m, |h|)$ with $|h|=1$ for $m=2k$ and $j=2,4,\ldots,2k$.
We can see that those boxes lie on the boundary of $\rho_m$. Similar results hold when $m=2k-1$ is odd.
Take $m=6$ and the shaded yellow boxes in the hooklength diagram of $\rho_6$ for example:
\begin{center}
\begin{ytableau}
11& *(yellow)9 & 7  &  *(yellow)5& 3 & *(yellow)1\\
 *(yellow)9& 7 & *(green)5 & 3 &*(yellow)1\\
7& *(green)5 & 3 & *(yellow)1\\
 *(yellow)5& 3 & *(yellow)1 \\
3& *(yellow)1\\
 *(yellow)1
\end{ytableau}\\
Figure 1: The hooklength diagram of $\rho_6 $
\end{center}

In the following lemma, we use notions in Chapter 1 of \cite{Olsson}.  A $\beta$-set for $\lambda$ is usually taken by the first column hooklength of $\lambda$. That is, if
$X_{\lambda}=\{x_1,x_2,\ldots,x_t\}$ is a $\beta$-set for $\lambda=(\lambda_1,\lambda_2,\ldots,\lambda_t)$, then $x_i=\lambda_i+t-i$ for $i=1,2,\ldots,t$. For each $\beta$-set $X=\{x_1,x_2,\ldots,x_t\}$, it corresponds to a unique partition $\mu$ which is defined by
$$\mu=(x_1-(t-1),x_2-(t-2),\ldots,x_t).$$
Let $H_{i}(\lambda)$ denote the set of hooklengths in row $i$ of $\lambda$.
\begin{lemma}\cite[Cor. 1.5]{Olsson}\label{lem:betah}
If $X=\{x_1,x_2,\ldots,x_t\}$ is a $\beta$-set for $\lambda$ and $r\in \mathbb{N}$, then
$$r\in H_{i}(\lambda)~\Leftrightarrow~x_i-r\geq0~and~x_i-r\notin X.$$
\end{lemma}

A box $(i,\lambda_i)\in\lambda$ is called a {\it removable} box (for $\lambda$)  if $\lambda_i>\lambda_{i+1}$. A box $(i,\lambda_i+1)$
is called {\it addable} (for $\lambda$)  if $i=1$ or $i>1$ and $\lambda_i<\lambda_{i-1}$.
We can see that a box is removable if and only if the corresponding hooklength is 1.
For $h\in H(\lambda)$, if $|h|=1$, $N(\lambda,|h|)$ is the set of all partitions (include $\lambda$) formed by moving out a removable box of $\lambda$ and then putting an addable box on the resulting partition. In this case, we denote  $N(\lambda,|h|)$ by $N(\lambda,1)$.

\begin{proposition}\label{prp:neigb}
 Let $\rho=(\rho_1,\rho_2,\ldots,\rho_k)$ be a partition in $P(n)$. Then for each $\tau\in N(\rho,1)$  we have $\chi^\tau(\hat{\rho})=0$ or $\pm1$.
\end{proposition}
\begin{proof}
Suppose that the Durfee size of $\rho$ is $d$. Let $\hat{\rho}=(\hat{\rho}_1,\hat{\rho}_2,\ldots,\hat{\rho}_d)$ be the principal hook partition of $\rho$.
Firstly, we consider the partition $\rho^{(i,1)}\in N(\rho,1)$ which is formed by putting the removable box in row $i$ ($i\geq2$) of $\rho$ to the addable box $(1,\rho_1+1)$ in the first row. Let $X_{\rho}=\{h_{1},h_{2},\ldots,h_{k}\}$ be the $\beta$-set of $\rho$ which is formed by the first column hooklengths. Then we have
$$h_i=\rho_i+k-i,$$
where $i=1,2,\ldots,k$.

Let $X^{(i,1)}$ be the $\beta$-set of $\rho^{(i,1)}$. Then we have
$$X^{(i,1)}=\{h_{1}+1,\ldots,h_{i-1},h_{i}-1,h_{i+1},\ldots,h_{k}\}.$$
By definition we know that $h_1=\hat{\rho}_1$ and $\hat{\rho}_1>h_2$. By Lemma \ref{lem:betah}, if  $\hat{\rho}_1$ appear in the hooklength diagram of $\rho^{(i,1)}$, we should have $h_k>1$. Moreover, there exists at most one hook with length $\hat{\rho}_1$ in  the hooklength diagram of $\rho^{(i,1)}$  and it must appear in the first row.

If there is one hook with length $\hat{\rho}_1$ denoted by $h(\hat{\rho}_1)$, then by \cite[Prop. 1.8]{Olsson} the resulting $\beta$-set after removing the hook is
$$X^{(i,1)}=\{h_{2},\ldots,h_{i-1},h_{i}-1,h_{i+1},\ldots,h_{k},1\}.$$
The partition corresponding to $X^{(i,1)}$ is
$$\rho^{(i,1)}_1=(\rho_2-1,\rho_3-1,\ldots,\rho_{i}-2,\ldots,\rho_k-1,1).$$
Let $\hat{\rho}\backslash \hat{\rho}_1=(\hat{\rho}_2,\hat{\rho}_3,\ldots,\hat{\rho}_d)$. By the M-N rule we have
$$\chi^{\rho^{(i,1)}}(\hat{\rho})=(-1)^{\ell(h(\hat{\rho}_1))}\chi^{\rho^{(i,1)}_1}(\hat{\rho}
\backslash \hat{\rho}_1).$$

The partition that is formed by removing the first principal hook of $\rho$ is
$$\rho\backslash\hat{\rho}_1=(\rho_2-1,\rho_3-1,\ldots,\rho_{i}-1,\ldots,\rho_k-1).$$
Then $\rho^{(i,1)}_1$ can be viewed as the partition which is obtained by putting the removable box in row $i-1$ of $\rho\backslash\hat{\rho}_1$ to the end (the $k$th row). By taking transpose, we can see that $\rho^{(i,1)'}_1$ can be viewed as the partition which is obtained by putting the removable box in some row of $(\rho\backslash\hat{\rho}_1)^{'}$ to the  first row. By similar discussions as above, we have that there exists at most  one hook with length $\hat{\rho}_2$ in the hooklength diagram of $\rho^{(i,1)'}_1$ and therefore  $\rho^{(i,1)}_1$. Hence, if we remove the hooklength $\hat{\rho}_1,\hat{\rho}_2,\ldots,\hat{\rho}_d$ decreasingly from $\rho^{(i,1)}$, then in each step $i$ there exists at most one hook with length $\hat{\rho}_i$. So by the M-N rule, we have that $\chi^{\rho^{(i,1)}}(\hat{\rho})=0$ or $\pm1$.

Secondly, suppose that  $\tau\in N(\rho,1)$ is formed by putting the removable box in row $i$  of $\rho$ to the addable box $(j,\rho_j+1)$ in row $j$. By taking transpose and removing the hooklength $\hat{\rho}_1,\hat{\rho}_2,\ldots,\hat{\rho}_d$ decreasingly, the discussion can be attributed to the conditions above.
\end{proof}

\begin{corollary}\label{cor:neigb1}
 Let $\rho_m=(m,m-1,\ldots,1)$. Then for each $\tau\in N(\rho_m,1)$  we have $\chi^\tau(\hat{\rho}_m)=0$ or $\pm1$. Moreover,  $\chi^\tau(\hat{\rho}_m)=\pm1$ if and only if $\hat{\tau}=\hat{\rho}_m$.
\end{corollary}
\begin{proof}
By Proposition \ref{prp:neigb}, we have that $\chi^\tau(\hat{\rho}_m)=0$ or $\pm1$ for each $\tau\in N(\rho_m,1)$.

Firstly, we consider the partition $\rho^{(i,1)}\in N(\rho_m,1)$ which is formed by putting the removable box in row $i$ ($i\geq2$) of $\rho_m$ to the addable box $(1,m+1)$ in the first row. Then the $\beta$-set of $\rho^{(i,1)}$ is
\begin{equation*}\label{eq:betai1}
\begin{split}
X^{(i,1)} & =\{x^{(i,1)}_{1},x^{(i,1)}_{2},\ldots,\}\\
&=\left\{ \begin{aligned}
&\{2m,2m-3,\ldots,2(m-i),\ldots,1\}\quad\text{if}~2\leq i\leq m-1\\
&\{2m-1,2(m-2),\ldots,4,2\}\quad\qquad\qquad\text{if}~i=m
        \end{aligned} \right..
\end{split}
\end{equation*}

Hence, if the hooklength of $2m-1$ appear in $H(\rho^{(i,1)})$, by Lemma \ref{lem:betah} we should have $i=m$. For $2\leq i\leq m-1$ since the rim hook of length $2m-1$ cannot appear in $H(\rho^{(i,1)})$, by the M-N rule we have $\chi^{\rho^{(i,1)}}(\hat{\rho}_m)=0$ where $\hat{\rho}_m=(2m-1,2m-5,\ldots)$. When $i=m$, we can see that $\hat{\rho}^{(m,1)}=\hat{\rho}_m$ and then $\chi^{\rho^{(m,1)}}(\hat{\rho}_m)=1$ or $-1$.

Generally, let $\rho^{(i,j)}\in  N(\rho_m,1)$ be the partition which is formed by putting the removable box in row $i$ of $\rho_m$ to the addable box in row $j$ where $1\leq i, j\leq m+1$.  By the M-N rule, the calculation of $\chi^{\rho^{(i,j)}}(\hat{\rho}_m)$ can be attributed to the condition in previous paragraph by taking transpose and removing the length of principal hooks of $\rho_m$ decreasingly.
\end{proof}

\begin{remark}
For general partition $\lambda$ and $\tau\in N(\lambda,1)$, if $\chi^\tau(\hat{\lambda})=\pm1$,  then $\hat{\tau}$ needn't have to equal $\hat{\lambda}$. For example $\lambda=(3,3,2)$ and $\tau=(4,2,2)$.

Suppose $\lambda=(\lambda_1,\lambda_2,\ldots,\lambda_l)\vdash n$ and $\mu=(\mu_1,\mu_2,\ldots,\mu_m)\vdash n$  are
partitions of $n$. Then $\lambda$ dominates $\mu$, written $\mu\trianglelefteq\lambda$, if $\lambda_1+\lambda_2+\cdots +\lambda_i \geq \mu_1 + \mu_2 +\cdots+\mu_i$
for all $i\geq1$ \cite[Def. 2.2.2]{sagan}. The dominate relation `$\trianglelefteq$' defines a partial order on $P(n)$ which is called the `dominance order'. For $\lambda,\mu$$\in P(n)$, we say that $\lambda$ is covered by $\mu$ (or $\mu$ covers $\lambda$), if $\lambda\trianglelefteq\mu$ and there is no  $\tau\in P(n)$ with $\lambda\trianglelefteq\tau\trianglelefteq\mu$.

Combining with \cite[Thm. 1.4.10]{JKer81} and Proposition \ref{prp:neigb}, we have that for all $\tau$ covers $\rho_m$, $\chi^{\tau}(\hat{\rho}_m)=0$ besides when $m=2k$ and move the removable box in row $k+1$ to the addable box in row $k$, and vice versa. So we can see that even two partitions are `closer' enough under the dominance order, one corresponding nonzero character value doesn't imply another.

In \cite{Ikenm}, Ikenmeyer showed that all those irreducibles which correspond to partitions that comparable to  $\rho_m$ in the dominance order satisfy Saxl's conjecture. By Corollary \ref{cor:neigb1} we have that there exist partitions $\tau$ that are comparable to  $\rho_m$ but $\chi^\tau(\hat{\rho}_m)=0$.
\end{remark}

\begin{corollary}\label{cor:only1}
Suppose that $h\in H(\rho_m)$ with $|h|=1$. Besides when $m=2k-1$ ($k\in \mathbb{N}$) and $h$ lies on row $k$ of $\rho_m$, there exists a unique $\tau\in N(\rho_m,1)$ with $\tau\neq\rho_m$ such that $\chi^{\tau}(\hat{\rho}_m)\neq0$ and $\chi^{\tau}\in\chi^{\rho_m}\otimes\chi^{\rho_m}$.
\end{corollary}

\begin{proof}
By Proposition \ref{cor:neigb1}, $\chi^{\tau}(\hat{\rho}_m)\neq0$ if and only if
$\hat{\tau}=\hat{\rho}_m$. Besides when $m=2k-1$ and $h$ lies on row $k$ of $\rho_m$, we have that  for each removable box $a_i$ in row $i$ of $\rho_m$, there exists only one addable box denoted by $b_j$ in some row $j$ such that the partition $\tau$ obtained by moving $a_i$ to $b_j$ satisfies  $\tau\neq\rho_m$ and $\hat{\tau}=\hat{\rho}_m$. Moreover, by Lemma 1.3 of \cite{PPV} we have  $\chi^{\tau}\in\chi^{\rho_m}\otimes\chi^{\rho_m}$.
\end{proof}

In the following discussion, we assume that $m=2k$ ($k\in \mathbb{N}$). Just like Corollary \ref{cor:only1}, we will see that the uniqueness of $|h|$-neighborhood also holds on the first row and column of $\rho_m$.

Before proving next lemma, we introduce a process (called `going around' in \cite{Regev}) which forms ~$N(\rho_m,  |h_{1,j}|)$. By definition, we can see that elements in $N(\rho_m,  |h_{1,j}|)$ are formed by rim hooks of length $2(m-j)+1$ which go around the boundary of $\rho_{1,j}=\rho_m\backslash h_{1,j}$ step by step. Taking $j=2$ for example, $\rho_{1,2}=\rho_{m}\setminus h_{1,2}=(2k-2,2k-3,\ldots,2,1,1,1)\vdash k(2k+1)-(4k-3)$. Firstly, let $\rho_{1,2}^{(1)}=(2k-2+(4k-3),~2k-3,\ldots,2,1,1,1)\vdash k(2k+1)$. This is the partition formed by putting all $4k-3$ boxes on the first row of $\rho_{1,2}$. In the second step, these added  $4k-3$ boxes begin to walk along the boundary of $\rho_{1,2}$ down and left with least moving boxes such that the result is a partition. Denote this partition by $\rho_{1,2}^{(2)}$. If these connected boxes continue to walk along the boundary of $\rho_{1,2}$ down and left such that a partition is formed by moving the least number of boxes, then we can get a series of partitions  $\rho_{1,2}^{(i)}$ in each step $i=1,2,\ldots$. By definition, we know that $N(\rho_m,  |h_{1,2}|)$ is formed by the $4k-3$ boxes that walk from the first row to the end of the first column of  $\rho_{1,2}$ step by step.

 Generally, for any partition $\lambda\vdash n$ and hook $h\in H(\lambda)$, $N(\lambda, |h|)$ can be obtained by the process described above. We can order elements in $N(\lambda, |h|)$ by the number of steps. In previous paragraph, $\rho_{1,2}^{(1)}$ in the first step is called {\it the initial partition}. Usually, the choice of the initial partition and moving direction depends on what we need. Taking $m=4$ for example:
\ytableausetup{smalltableaux}
\begin{align*}
\ydiagram[*(white)]
{2,1,1,1}
*[*(yellow)]{2,1,1,1,1,1,1,1,1}
\leftarrow
\ydiagram[*(white)]
{2,1,1,1}
*[*(yellow)]{2,2,2,2,2}
&\leftarrow
\ydiagram[*(yellow)]
{2+1,1+2,1+1,1+1}
*[*(white)]{3,3,2,2}
\leftarrow
\ydiagram[*(yellow)]
{2+2,1+2,1+1}
*[*(white)]{4,3,2,1}
\rightarrow
\ydiagram[*(yellow)]
{2+3,1+2}
*[*(white)]{5,3,1,1}
\rightarrow
\ydiagram[*(yellow)]
{2+5}
*[*(white)]{7,1,1,1}\\
& \text{Figure 2: The going around process}
\end{align*}
Choosing $\rho_4$ as the initial partition, in Figure 2 above, $N(\rho_4, |h_{1,2}|)$ is depicted by the going around process. The shaded  are the moving boxes in each step. The left arrow shows the process of moving down and left, the right arrow shows the process of moving up and right.

\begin{lemma}\label{lem:khook}
Let $h_{1,j}$ ($j=2,4,\ldots,2k$) be the hooks on the first row of $\rho_m$ and  $\rho_{1,j}=\rho_m\setminus h_{1,j}$. Then
for $\tau\in N(\rho_m, |h_{1,j}|)$,  $4k-1\in H(\tau)$ if and only if $\tau=\rho_m$ or $\tau_{1,j}$ where $\tau_{1,j}$ is the partition obtained by moving all the $|h_{1,j}|$ boxes on the end of the first column of $\rho_{1,j}$.
 \end{lemma}
 \begin{proof}
For each~$\tau\in N(\rho_m, |h_{1,j}|)$, it is not hard to see that the hooks with length $4k-1$ must appear in the first row or column. Just as the `going around process' described in Figure 2 we choose the initial partition as $\rho_m$. Then $N(\rho_m, |h_{1,j}|)$ is formed by the rim hook corresponding to  $h_{1,j}$ moves up to the first row and down to the end of  $\rho_{1,j}$. We have that if $\tau\in N(\rho_m, |h_{1,j}|)$ belongs to the moving up (down, respectively) process then  $4k-1\in H(\tau)$ if and only if $4k-1$ appear in the first row (column, respectively), that is, $H_1(\tau)$.

If $\tau\in N(\rho_m, |h_{1,j}|)$ belongs to the moving up process, then we can use Lemma \ref{lem:betah} to determine whether $4k-1$ appear or not. In fact, let $X_\tau=\{x_{\tau_1},x_{\tau_2},\ldots,x_{\tau_l}\}$ denote the $\beta$-set formed by the first column hooklength. We have that if $\tau\neq \rho_m$ then $x_{\tau_1}-(4k-1)\in X_{\tau}$, and so by Lemma \ref{lem:betah} we have $4k-1\notin H_1(\tau)$. The moving down process turns out to move up if we take transpose. Since
taking transpose does not change the hooklength for each box, the discussion of
$\tau\in N(\rho_m, |h_{1,j}|)$ that belongs to the moving down process can be attributed to the condition of moving up. In this case, we will find that $x_{\tau_1}-(4k-1)\notin X_{\tau}$ only if $\tau=\tau_{1,j}$, that is, the partition obtained by moving the hook $h_{1,j}$ down till the end of the first column. Hence we have $4k-1\in H(\tau_{1,j})$.

More precisely, we can deduce $x_{\tau_1}-(4k-1)$  as follows:
choosing the initial partition as $\rho_m$, we can order partitions in  $N(\rho_m, |h_{1,j}|)$ by the number of step. Denote $\rho_m^{(0)}=\rho_m$. The partitions corresponding to the moving up direction in step $i$ is denoted by $\rho_m^{(i)}$. The  tail column hooklength in step $i$ refers to the first column hooklengths that under the last lower left box of the moving rim hook. Let  $T^i$ denote the tail column hooklength corresponding to $\rho_{m}^{(i)}$. With notions above, we can directly check that if $\tau=\rho_{m}^{(i)}$ ($i>0$) which corresponds to the moving up process then $x_{\tau_1}-(4k-1)\in T^i\subseteq X_{\tau}$. Hence by Lemma \ref{lem:betah} we have $4k-1\notin H_{1}(\tau)$ and so $4k-1\notin H(\tau)$. For example, in Figure 2, we have $\rho_4^{(1)}=(5,3,1,1)$ and $T^1=\{2,1\}$. Let $\tau=\rho_4^{(1)}$, then $X_\tau=\{x_{\tau_1},x_{\tau_2},\ldots\}=\{8,5,2,1\}$. Then $x_{\tau_1}-(4k-1)=8-7=1\in T^1$, where $k=2$. So if we check the partitions in Figure 2 by the $\beta$-set, we have that  besides $\rho_4$ there is a unique $\tau_{1,2}\in N(\rho_4,|h_{1,2}|)$ such that $7\in H(\tau_{1,2})$, where $\tau_{1,2}=(2,1^8)$. We can see that $\tau_{1,2}$ is the leftmost partition
obtained by moving all the $|h_{1,2}|=5$ boxes on the end of the first column of $\rho_{1,2}=(2,1^3)$.
\end{proof}

Since $\rho_m$ is self-conjugate, similar result holds for hooks on the first column.
\begin{corollary}\label{cor:khookc}
Let $h_{j,1}$ ($j=2,4,\ldots,2k$) be the hooks on the first column of $\rho_m$ and  $\rho_{j,1}=\rho_m\setminus h_{j,1}$. Then
for $\tau\in N(\rho_m, |h_{j,1}|)$,  $4k-1\in H(\tau)$ if and only if $\tau=\rho_m$ or $\tau_{j,1}$ where $\tau_{j,1}$ is the partition obtained by moving all the $|h_{j,1}|$ boxes on the end of the first row of $\rho_{j,1}$.
\end{corollary}

In the following theorem, $\tau_{1,j}$ are the partitions obtained in Lemma \ref{lem:khook}. We will provide a class of nonzero Kronecker coefficients corresponding to Saxl's conjecture by Pak et al.'s main lemma.
\begin{theorem}\label{thm:rcnb}
For the hooks  $h_{1,j}$ where $j=2,4,\ldots,2k$ on the first row of $\rho_m$ with lengths $\{4k-3,4k-7,\ldots,5,1\}$, there exists a unique $\tau_{1,j}\in N(\rho_m, |h_{1,j}|)$ with $\tau_{1,j}\neq\rho_m$ such that $\chi^{\tau_{1,j}}(\hat{\rho}_m)\neq0$ and $\chi^{\tau_{1,j}}\in\chi^{\rho_m}\otimes\chi^{\rho_m}$.
\end{theorem}

\begin{proof}
Since  $4k-1$ is the first part of $\hat{\rho}_m=(4k-1,4k-5,\ldots,7,3)$, by the M-N rule and Lemma \ref{lem:khook} we have $\chi^{\lambda}(\hat{\rho}_m)=0$ for all $\lambda\in N(\rho_m, |h_{1,j}|) $ besides $\rho_m$ and $\tau_{1,j}$. Moreover,  since $\chi^{\rho_m}(\hat{\rho}_m)=(-1)^{k(k-1)/2}$, from the argument of Theorem \ref{thm:neib} we know that $\chi^{\tau_{1,j}}(\hat{\rho}_m)=1$ or $-1$. Hence, by Lemma 1.3 of \cite{PPV} we have  $\chi^{\tau_{1,j}}\in \chi^{\rho_m}\otimes\chi^{\rho_m}$.
\end{proof}

\begin{remark}
In Theorem \ref{thm:rcnb}, we only discuss the condition when the hooks  $h_{1,j}$ on the first row of $\rho_m$. Similar arguments hold for $h_{j,1}$ on the first column, where $j=2,4,\ldots,2k$. If $m=2k-1$ ($k\in \mathbb{N}$), by similar discussions as Theorem \ref{thm:rcnb} we can also find a unique neighborhood for each $h\in H(\rho_m)$ with $|h|\notin \{4k-3,4k-7,\ldots,5,1\}$ on the first row and column of $\rho_m$.

Suppose that $h\in H(\rho_m)$  with hooklength $|h|$ does't belong to principal hooklengths of $\rho_m$ and $h$ does't lie on the boundary of $\rho_m$, such as the boxes with green color in Figure 1. For smaller $n$, by direct calculation we have that the result of Theorem \ref{thm:rcnb} still holds. That is, there exists only one $\tau\in N(\rho_m, |h|)$ such that $\tau\neq\rho_m$, $\chi^\tau(\hat{\rho}_m)\neq0$ and therefore $\chi^\tau(\hat{\rho}_m)=\pm1$. Generally, we have the following problem.

\begin{problem}
 Suppose that $\lambda\in P(n)$. For each $h\in H(\lambda)$ with hooklength doesn't belong to the principal hooklengths of $\lambda$, does there exist a unique partition $\nu\in N(\lambda,|h|)$ such that $\nu\neq\lambda$ and $\chi^\nu(\hat{\lambda})\neq0$?
\end{problem}
\end{remark}

\begin{remark}
 It is not hard to see that the partitions obtained in Corollary \ref{cor:only1} and Theorem  \ref{thm:rcnb} can also be detected by Ikenmeyer's criterion \cite[Thm. 2.1]{Ikenm}. Moreover, there are many irreducible characters in the neighborhood of $\rho_m$ that take zero on $\hat{\rho}_m$, which reflects that Ikenmeyer's  criterion may be more powerful than Pak et al.'s for Saxl's conjecture. The method we used here can also be applicable to other self-conjugate partitions. In next section, we will discuss the effectiveness of the character criterion for another self-conjugate partition.
\end{remark}

\section{Effectiveness of the character criterion by an example}\label{se:effectiveness}
Let $\gamma=(k+1,2,1^{k-1})$ be self-conjugate where $k\in \mathbb{N}$. Then the Durfee size of $\gamma$ is two. In this section, we will find the nonzero Kronecker coefficients of $\chi^{\gamma}\otimes \chi^{\gamma}$ explicitly. Then we have that the character criterion is less effective for large $k$. Due to the Frobenius correspondence between character and Schur function, in this section we discuss Kronecker coefficients under the expression of Schur functions.

In \cite{Lwood}, Littlewood provided the following theorem about the mixed product of ordinary and Kronecker product of Schur functions.
\begin{theorem}\cite[Thm.~3]{Lwood}\label{thm:littlew}
 Let $\alpha,~\beta$ and $\gamma$ be partitions such that $|\alpha| +|\beta| =|\gamma|$. Then,
$$(s_\alpha s_\beta)*s_\gamma =\sum_{\eta\vdash|\alpha|}\sum_{\delta\vdash|\beta|}c_{\eta,\delta}^{\gamma}(s_\eta *s_\alpha)(s_\delta * s_\beta),$$
where  $c_{\eta,\delta}^{\gamma}$ are Littlewood-Richardson coefficients.
\end{theorem}
\begin{remark}
Littlewood stated that `repeated application of this theorem will evaluate any inner product (Kronecker product) of S-functions'. Generally, the evaluation is complicated. But in some special cases this formula is practicable, for example \cite{Tewari}. In this section we will provide another example.
\end{remark}

For a partition $\gamma$, let $\gamma\ominus1$ (respectively, $\gamma\oplus1$ ) denote the set of partitions obtained by removing a removable box of $\gamma$ (respectively, putting an addable box of $\gamma$). By Littlewood-Richardson rule \cite{sagan} we know that if
$c_{\eta,1}^{\gamma}$ is nozero, then $c_{\eta,1}^{\gamma}=1$ and $\eta\in\gamma\ominus1$.
In Theorem \ref{thm:littlew}, if we let $\beta=1$ we have
\begin{equation}\label{eq:littlew}
\begin{split}
(s_\alpha s_1)*s_\gamma &=\sum_{\eta\vdash|\alpha|}\sum_{\delta\vdash 1}c_{\eta,\delta}^{\gamma}(s_\eta *s_\alpha)(s_\delta * s_1)=\sum_{\eta\vdash|\alpha|}c_{\eta,1}^{\gamma}(s_\eta *s_\alpha)(s_1 * s_1),\\
&=\sum_{\eta\in\gamma\ominus1}(s_\eta *s_\alpha)s_1.
\end{split}
\end{equation}

The following theorem is a special case of the Littlewood-Richardson rule which is called the {\it Pieri rule}.
\begin{theorem}\label{thm:pierule}
If $\mu$ is a partition, then
$$s_\mu s_{(n)}= \sum_{\nu\vdash |\mu|+n}s_\nu,$$
where $\nu/\mu$ is a horizontal strip of size $n$.
\end{theorem}

In Theorem \ref{thm:pierule}, if we let $n=1$ we have
\begin{equation}\label{eq:pie1}
 s_\mu s_{1}= \sum_{\nu\in \mu\oplus1}s_\nu.
\end{equation}
The Kronecker product also satisfies the useful identities
\begin{equation}\label{eq:krkp}
s_\alpha*s_\beta= s_\beta * s_\alpha~\text{and}~s_{\alpha\prime} * s_{\beta\prime} =s_\alpha*s_\beta.
\end{equation}

Let $\gamma=(k+1,2,1^{k-1})\vdash 2k+2$ be self-conjugate.  We will discuss the decomposition of $s_\gamma* s_\gamma$. In the following discussion, the notation of the following partitions are fixed: $\tau=(k+2,1^k)$, $\alpha=(k+1,1^k)$, $\eta=(k+1,2,1^{k-2})$, $\tilde{\alpha}=(k+2,1^{k-1})$ and $\alpha^-=(k+1,1^{k-1})$.

In the following proposition, by (\ref{eq:littlew}) we get an expression of $ s_\gamma* s_\gamma$ in a relative `simpler' forms under the mixed product of hook partitions and $s_1$. We will use it to evaluate $ s_\gamma*s_\gamma$.
\begin{proposition}\label{prop:srsr}
 For $\gamma$, $\tau$, $\alpha$, $\eta$, $\tilde{\alpha}$ and $\alpha^-$ defined above, we have
\begin{equation}\label{eq:srsr}
\begin{split}
  s_\gamma* s_\gamma= & 2(s_{\alpha^-}* s_{\alpha^-})s_1s_1+2(s_{\alpha^{-\prime}}* s_{\alpha^-})s_1s_1-3(s_{\alpha}* s_{\alpha})s_1-4(s_{\tilde{\alpha}}* s_{\alpha})s_1\\
    &+2s_\tau*s_\tau+2s_{\tau\prime}*s_\tau,
\end{split}
\end{equation}
where $\alpha^{-\prime}$, $\tau^\prime$ are conjugate partitions of  $\alpha^{-}$ and $\tau$.
\end{proposition}

\begin{proof}
By (\ref{eq:littlew}) we have
$(s_\alpha s_1)*s_\gamma=\sum_{\beta\in\gamma\ominus1}(s_\alpha*s_\beta)s_1$.
Since $\gamma$ and $\alpha$ are self-conjugate, by (\ref{eq:pie1}) and (\ref{eq:krkp}) we have
\begin{align*}
 (s_\alpha s_1)*s_\gamma &=(s_\tau+s_\gamma+s_{\tau\prime})*s_\gamma
=(s_\tau+s_{\tau\prime})*s_\gamma+s_\gamma*s_\gamma\\
&=2s_\tau*s_\gamma+s_\gamma*s_\gamma
\end{align*}
and
\begin{align*}
\sum_{\beta\in\gamma\ominus1}(s_\alpha*s_\beta)s_1
&=(s_\alpha*s_\eta+s_\alpha*s_\alpha+s_\alpha*s_{\eta\prime})s_1\\
&=(2s_\alpha*s_\eta+s_\alpha*s_\alpha)s_1.
\end{align*}
Hence, we have
\begin{equation}\label{eq:srsrgg}
s_\gamma*s_\gamma=(2s_\alpha*s_\eta+s_\alpha*s_\alpha)s_1- 2s_\tau*s_\gamma.
\end{equation}

Similarly, we have
$(s_{\alpha^-}s_1)*s_\alpha=\sum_{\beta\in\alpha\ominus1}
(s_{\alpha^{-}}*s_\beta)s_1$ and so
$$(s_{\tilde{\alpha}}+s_\alpha+s_\eta)*s_\alpha
=(s_{\alpha^{-}}*s_{\alpha^{-}})s_1+(s_{\alpha^{-}\prime}*s_{\alpha^{-}})s_1.$$
Hence,
\begin{equation}\label{eq:srsrea}
s_\eta*s_\alpha=(s_{\alpha^{-}}*s_{\alpha^{-}})s_1+(s_{\alpha^{-}\prime}*s_{\alpha^{-}})s_1-
s_{\tilde{\alpha}}*s_\alpha-s_\alpha*s_\alpha.
\end{equation}

Since $(s_{\alpha}s_1)*s_\tau=\sum_{\beta\in\tau\ominus1}
(s_{\alpha}*s_\beta)s_1$, for $s_\tau*s_\gamma$ we have
\begin{equation}\label{eq:srsrgt}
s_\tau*s_\gamma=(s_{\alpha}*s_{\alpha})s_1+(s_{\tilde{\alpha}}*s_{\alpha})s_1-
s_{\tau}*s_\tau-s_{\tau\prime}*s_\tau.
\end{equation}

Combining (\ref{eq:srsrgg}), (\ref{eq:srsrea}) and (\ref{eq:srsrgt}), we get (\ref{eq:srsr}).
\end{proof}

\begin{remark}
Generally, the method in the proof of Proposition \ref{prop:srsr} can be used to evaluate $ s_\lambda*s_\mu$ with principal hook partition $\hat{\lambda}=\hat{\mu}=(k,1)$.
\end{remark}

For each $\lambda\in P(2k+2)$, denote the multiplicity of $s_\lambda$ in $(s_{\alpha^{-}}*s_{\alpha^{-}})s_1s_1$,  $(s_{\alpha^{-\prime}}*s_{\alpha^{-}})s_1s_1$,
$(s_{\alpha}*s_{\alpha})s_1$ and
$(s_{\tilde{\alpha}}*s_{\alpha})s_1$
by $g^{(2)}(\alpha^{-},\alpha^{-},\lambda)$, $g^{(2)}(\alpha^{-\prime},\alpha^{-},\lambda)$,
$g^{(1)}(\alpha,\alpha,\lambda)$ and
$g^{(1)}(\tilde{\alpha},\alpha,\lambda)$, respectively.
Then we have

\begin{equation*}
g^{(1)}(\alpha,\alpha,\lambda)=\sum_{\sigma\in \lambda\ominus1}g(\alpha,\alpha,\sigma) \quad and\quad g^{(2)}(\alpha^{-},\alpha^{-},\lambda)=\sum_{\sigma\in \lambda\ominus1} \sum_{\nu\in\sigma\ominus1}g(\alpha^{-},\alpha^{-},\nu).
\end{equation*}

Similar results hold for $g^{(2)}(\alpha^{-\prime},\alpha^{-},\lambda)$
and
$g^{(1)}(\tilde{\alpha},\alpha,\lambda)$.
Then by (\ref{eq:srsr}) we have
\begin{equation}\label{eq:gmgmlb}
\begin{split}
g(\gamma,\gamma,\lambda)= & 2g^{(2)}(\alpha^{-},\alpha^{-},\lambda)+2g^{(2)}(\alpha^{-\prime},\alpha^{-},\lambda)
-3g^{(1)}(\alpha,\alpha,\lambda)-4g^{(1)}(\tilde{\alpha},\alpha,\lambda)\\
    &+2g(\tau,\tau,\lambda)+2g(\tau',\tau,\lambda)\\
=&2\sum_{\sigma\in \lambda\ominus1}\sum_{\nu\in\sigma\ominus1}g(\alpha^{-},\alpha^{-},\nu)
+2\sum_{\sigma\in\lambda\ominus1}
\sum_{\nu\in\sigma\ominus1}g(\alpha^{-\prime},\alpha^{-},\nu)-3\sum_{\sigma\in \lambda\ominus1} g(\alpha,\alpha,\sigma)\\
&-4\sum_{\sigma\in \lambda\ominus1} g(\tilde{\alpha},\alpha,\sigma)+2g(\tau,\tau,\lambda)+2g(\tau',\tau,\lambda)\\
=&\sum_{\sigma\in \lambda\ominus1}\left(\sum_{\nu\in\sigma\ominus1}2\left(g(\alpha^{-},\alpha^{-},\nu)+ g(\alpha^{-\prime},\alpha^{-},\nu)\right)-3 g(\alpha,\alpha,\sigma)-4g(\tilde{\alpha},\alpha,\sigma)\right)\\
&+2g(\tau,\tau,\lambda)+2g(\tau',\tau,\lambda).
\end{split}
\end{equation}

In \cite{rosas}, Rosas discussed the decomposition of Kronecker product of Schur functions indexed by hook shapes. By Theorem 3 of  \cite{rosas} for $\tau=(k+2,1^k)$, $\alpha=(k+1,1^k)$, $\tilde{\alpha}=(k+2,1^{k-1})$ we have the following Table \ref{tab:kchk}.

\begin{table}[h]
 \begin{tabular}{|c|c|c|c|c|}\hline
   & $g(\tau,\tau,\lambda)$, & $g(\tau,\tau',\lambda)$ &  $g(\alpha,\alpha,\lambda)$, &  $g(\tilde{\alpha},\alpha,\lambda)$ \\\hline
  $\lambda=(a,1^b),~b>0$, $a=1$         & 0 & 1 & 1 & 0 \\\hline
 $\lambda=(a,1^b),~b>0$, $a>1$         & 1  & 1 & 1 & 1 \\\hline
  $\lambda=(a,1^b),~b=0$             & 1 & 0 & 1 & 0  \\\hline
$ \lambda=(a,b,2^c,1^d),~d=0,~a=b$ & 1 & 1 & 2 & 0 \\\hline
 $\lambda=(a,b,2^c,1^d),~d>0,~a=b$ & 1 & 2 & 2 & 1 \\\hline
   $\lambda=(a,b,2^c,1^d),~d=0,~a>b$ & 2 & 1 &2  &1  \\\hline
  $\lambda=(a,b,2^c,1^d),~d>0,~a>b$ & 2 & 2 & 2 & 2 \\\hline
  Others & 0 & 0 & 0 & 0 \\\hline
\end{tabular}
\caption{Kronecker coefficients for some hook partitions}
\label{tab:kchk}
\end{table}
\begin{remark}
 Since $|\tau|=2k+2$ and $|\alpha|=|\tilde{\alpha}|=2k+1$, in first column of Table \ref{tab:kchk} we just describe the shape of $\lambda$. So we have $\lambda\in P(2k+2)$ in column 2 and 3, in column 4 and 5 we have $\lambda\in P(2k+1)$.
In column 2 and 3, if we replace $k$ by $k-1$, then $\tau$ and $\tau'$ become $\alpha^-=(k+1,1^{k-1})$ and $\alpha^{-\prime}=(k,1^{k})$, and the results still hold.
\end{remark}

%
%
%


\begin{proposition}\label{prop:srsrds1}
For $\gamma=(k+1,2,1^{k-1})\vdash 2k+2$, suppose $s_\gamma*s_\gamma=\sum_{\lambda\in P(2k+2)}g(\gamma,\gamma,\lambda)s_\lambda$. If $\lambda\in DS(1,2k+2)$, then $g(\gamma,\gamma,\lambda)\neq0$. Moreover, suppose $\lambda=(a,1^b)\vdash 2k+2$ with $a>b$. Then we have
\begin{equation*}
\begin{split}
g(\gamma,\gamma,\lambda)&=\left\{\begin{aligned}
&1\qquad\text{if}\quad b=0\\
&2\qquad\text{if}\quad b=1\\
&4\qquad\text{if}\quad b=2\\
&6\qquad\text{if}\quad b>2
        \end{aligned} \right..
\end{split}
\end{equation*}
\end{proposition}

\begin{proof}
If $b=0$, it is well-known that $g(\gamma,\gamma,\lambda)=1$.
By Table \ref{tab:kchk} and (\ref{eq:gmgmlb}), $g(\gamma,\gamma,\lambda)$ can be directly evaluated for $b=1,2,\ldots$. Take $b=1$ for example.

If $b=1$, then $\lambda=(2k+1,1)$ and $\lambda\ominus1=\{(2k+1), (2k,1)\}$. Denote $\sigma_1=(2k+1)$ and $\sigma_2=(2k,1)$. Then $\sigma_1\ominus1=\{(2k)\}$ and $\sigma_2\ominus1=\{(2k-1,1),(2k)\}$.

By Table \ref{tab:kchk},
we have the following: $g(\alpha^{-},\alpha^{-},(2k))=1$,
$g(\alpha^{-\prime},\alpha^{-},(2k))=0$; $g(\alpha^{-},\alpha^{-},(2k-1,1))=g(\alpha^{-\prime},\alpha^{-},(2k-1,1))=1$;
$g(\alpha,\alpha,\sigma_1)=1, g(\tilde{\alpha},\alpha,\sigma_1)=0$;
$g(\alpha,\alpha,\sigma_2)=1, g(\tilde{\alpha},\alpha,\sigma_2)=1$;
$g(\tau,\tau,\lambda)=1, g(\tau',\tau,\lambda)=1$.

Hence by (\ref{eq:gmgmlb}) we have
\begin{equation*}
\begin{split}
g(\gamma,\gamma,\lambda)
=&2(g(\alpha^{-},\alpha^{-},(2k))+g(\alpha^{-\prime},\alpha^{-},(2k)))-3 g(\alpha,\alpha,\sigma_1)-4g(\tilde{\alpha},\alpha,\sigma_1)\\
&+2(g(\alpha^{-},\alpha^{-},(2k-1,1))+ g(\alpha^{-\prime},\alpha^{-},(2k-1,1)))\\
&+2(g(\alpha^{-},\alpha^{-},(2k))+ g(\alpha^{-\prime},\alpha^{-},(2k)))
-3g(\alpha,\alpha,\sigma_2)
-4g(\tilde{\alpha},\alpha,\sigma_2)\\
&+2g(\tau,\tau,\lambda)+2g(\tau',\tau,\lambda)\\
=&2
\end{split}
\end{equation*}

If $b=2$ (or $b>2$), similar arguments can be used to get $g(\gamma,\gamma,\lambda)=4$ (or 6).
\end{proof}

Since $g(\gamma,\gamma,\lambda)=g(\gamma,\gamma,\lambda')$, by Proposition \ref{prop:srsrds1} we have that $g(\gamma,\gamma,\lambda)>0$ for all $\lambda\in DS(1,2k+2)$.

\begin{theorem}\label{thm:srsrds2}
For $\gamma=(k+1,2,1^{k-1})\vdash 2k+2$, suppose $s_\gamma*s_\gamma=\sum_{\lambda\in P(2k+2)}g(\gamma,\gamma,\lambda)s_\lambda$. If $\lambda\in DS(2,2k+2)$, then $g(\gamma,\gamma,\lambda)>0$.
\end{theorem}

\begin{proof}
If $k=1$ and $2$, the computation (see also \cite{Lwood} for $k=2$) is directly by (\ref{eq:gmgmlb}). In the following discussion, we assume $k\geq3$.

If $\lambda\in DS(2,2k+2)$, we have $\hat{\lambda}=(\hat{\lambda}_1,\hat{\lambda}_2)$ with $\hat{\lambda}_2\geq1$. Next, we will show $g(\gamma,\gamma,\lambda)>0$ under conditions: (1) $\hat{\lambda}_2\geq3$; (2) $\hat{\lambda}_2=2$; (3) $\hat{\lambda}_2=1$. Let $$f(\sigma)=\sum_{\nu\in\sigma\ominus1}2\left(g(\alpha^{-},\alpha^{-},\nu)+ g(\alpha^{-\prime},\alpha^{-},\nu)\right)-3 g(\alpha,\alpha,\sigma)-4g(\tilde{\alpha},\alpha,\sigma).$$
Then
$g(\gamma,\gamma,\lambda)=\sum_{\sigma\in \lambda\ominus1}f(\sigma)
+2g(\tau,\tau,\lambda)+2g(\tau',\tau,\lambda)$ by (\ref{eq:gmgmlb}).
We will show that  $f(\sigma)\geq0$ under conditions (1), (2), (3) above. Since
$g(\tau,\tau,\lambda)$, $g(\tau',\tau,\lambda)>$0, we have $g(\gamma,\gamma,\lambda)>0$.

(1) Suppose that $\hat{\lambda}_2\geq3$. Then no partitions with Durfee size 1 appear in $\lambda\ominus1$ and $\sigma\ominus1$ where $\sigma\in\lambda\ominus1$. So we have $g(\alpha,\alpha,\sigma)=2$ and  $g(\tilde{\alpha},\alpha,\sigma)=1$ or $2$ by Table \ref{tab:kchk} (since $|\sigma|$ is odd, $g(\tilde{\alpha},\alpha,\sigma)\neq0$).

If  $g(\tilde{\alpha},\alpha,\sigma)=1$, then we have $\sigma$ or $\sigma'=(a,b,2^c,1^d)$ with $d=0$ and $a>b$. Hence $|\sigma\ominus1|\geq2$.
For each $\nu\in\sigma\ominus1$, the last row of $\nu$ has a removable box. So there exists $\nu_0\in \sigma\ominus1$ with the form $\nu_0=(x,y,2^s,1^t)$ where $x>y$ and $t>0$. Hence $g(\alpha^{-},\alpha^{-},\nu_0)=g(\alpha^{-'},\alpha^{-},\nu_0)=2$.  Since $g(\alpha^{-},\alpha^{-},\nu)$, $g(\alpha^{-\prime},\alpha^{-},\nu)\geq1$ for $\nu\in \sigma\ominus1$ and $|\sigma\ominus1|\geq2$, we have the following estimation
\begin{align*}
f(\sigma)=&\sum_{\nu\in\sigma\ominus1}2\left(g(\alpha^{-},\alpha^{-},\nu)+ g(\alpha^{-\prime},\alpha^{-},\nu)\right)-3 g(\alpha,\alpha,\sigma)-4g(\tilde{\alpha},\alpha,\sigma)\\
=&\sum_{\overset{\nu\in\sigma\ominus1}{\nu\neq\nu_0}}2\left(g(\alpha^{-},\alpha^{-},\nu)+ g(\alpha^{-\prime},\alpha^{-},\nu)\right)+2\left(g(\alpha^{-},\alpha^{-},\nu_0)+ g(\alpha^{-\prime},\alpha^{-},\nu_0)\right)-3\times2-4\times1\\
\geq&2(1+1)+2(2+2)-6-4\\
=&2>0.
\end{align*}

If  $g(\tilde{\alpha},\alpha,\sigma)=2$, then we have $\sigma$ or $\sigma'=(a,b,2^c,1^d)$ with $d>0$ and $a>b$. Hence $|\sigma\ominus1|\geq3$.
So there exists $\nu_0\in \sigma\ominus1$ with the form $\nu_0=(x,y,2^s,1^t)$ where $x>y$ and $t>0$. Hence $g(\alpha^{-},\alpha^{-},\nu_0)=2$ and $g(\alpha^{-\prime},\alpha^{-},\nu_0)=2$. Since  $g(\alpha^{-},\alpha^{-},\nu)$, $g(\alpha^{-\prime},\alpha^{-},\nu)\geq1$ for $\nu\in \sigma\ominus1$ and $|\sigma\ominus1|\geq3$, we have the following estimation
\begin{align*}
f(\sigma)=&\sum_{\nu\in\sigma\ominus1}2\left(g(\alpha^{-},\alpha^{-},\nu)+ g(\alpha^{-\prime},\alpha^{-},\nu)\right)-3 g(\alpha,\alpha,\sigma)-4g(\tilde{\alpha},\alpha,\sigma)\\
=&\sum_{\overset{\nu\in\sigma\ominus1}{\nu\neq\nu_0}}2\left(g(\alpha^{-},\alpha^{-},\nu)+ g(\alpha^{-\prime},\alpha^{-},\nu)\right)+2\left(g(\alpha^{-},\alpha^{-},\nu_0)+ g(\alpha^{-\prime},\alpha^{-},\nu_0)\right)-3\times2-4\times2\\
\geq&2(2+2)+2(2+2)-6-8\\
=&2>0.
\end{align*}

(2) Suppose that $\hat{\lambda}_2=2$. Then no partitions with Durfee size 1 appear in $\lambda\ominus1$ and there exists one partition with Durfee size 1 (i.e. hook partition) in $\sigma\ominus1$ for $\sigma\in\lambda\ominus1$.
So we have $g(\alpha,\alpha,\sigma)=2$ and  $g(\tilde{\alpha},\alpha,\sigma)=1$ or $2$ by Table \ref{tab:kchk} (since $|\sigma|$ is odd, $g(\tilde{\alpha},\alpha,\sigma)\neq0$).

If  $g(\tilde{\alpha},\alpha,\sigma)=1$, then we have either $\sigma$ or $\sigma'=(a,b,2^c,1^d)$ with $d=0$ and $a>b$. Hence $|\sigma\ominus1|\geq2$.
We only consider the case when $\sigma$ or $\sigma'=(a,2)$ with $a>2$. Otherwise,
$\sigma\ominus1$ contains no hook partitions and the discussion can be attributed to condition (1). Since $k\geq3$, we have $a\geq5$. Then
 there exists $\nu_0\in \sigma\ominus1$ with the form $\nu_0=(a-1,2)$. Hence $g(\alpha^{-},\alpha^{-},\nu_0)=2$ and we have the following estimation
\begin{align*}
f(\sigma)=&\sum_{\nu\in\sigma\ominus1}2\left(g(\alpha^{-},\alpha^{-},\nu)+ g(\alpha^{-\prime},\alpha^{-},\nu)\right)-3 g(\alpha,\alpha,\sigma)-4g(\tilde{\alpha},\alpha,\sigma)\\
=&\sum_{\overset{\nu\in\sigma\ominus1}{\nu\neq\nu_0}}2\left(g(\alpha^{-},\alpha^{-},\nu)+ g(\alpha^{-\prime},\alpha^{-},\nu)\right)+2\left(g(\alpha^{-},\alpha^{-},\nu_0)+ g(\alpha^{-\prime},\alpha^{-},\nu_0)\right)-3\times2-4\times1\\
\geq&2(1+1)+2(2+1)-6-4=0.
\end{align*}

If  $g(\tilde{\alpha},\alpha,\sigma)=2$, then we only discuss when $\sigma$ or $\sigma'=(a,2,1^d)$ with $d>0$ and $a>2$. Otherwise, we have $\sigma$ or $\sigma'=(a,3,1^d)$ and $\sigma\ominus1$ contains no hook. The discussion can also be attributed to condition (1).
Hence $|\sigma\ominus1|\geq3$.
For each $\nu\in\sigma\ominus1$, the first and last row of $\nu$ have removable boxes. So there exist $\tau_0$, $\nu_0\in \sigma\ominus1$ with the form $\tau_0=(a,2,1^b)$ where $a>2$ and $b\geq0$, $\nu_0=(c,2,1^d)$ where $c\geq2$ and $d>0$. Hence $g(\alpha^{-},\alpha^{-},\tau_0)=2$ and $g(\alpha^{-\prime},\alpha^{-},\nu_0)=2$. So we have the following estimation
\begin{align*}
f(\sigma)=&\sum_{\nu\in\sigma\ominus1}2\left(g(\alpha^{-},\alpha^{-},\nu)+ g(\alpha^{-\prime},\alpha^{-},\nu)\right)-3 g(\alpha,\alpha,\sigma)-4g(\tilde{\alpha},\alpha,\sigma)\\
=&\sum_{\overset{\nu\in\sigma\ominus1}{\nu\neq\tau_0,\nu_0}}2\left(g(\alpha^{-},\alpha^{-},\nu)+ g(\alpha^{-\prime},\alpha^{-},\nu)\right)+2\left(g(\alpha^{-},\alpha^{-},\tau_0)+ g(\alpha^{-\prime},\alpha^{-},\tau_0)\right)\\
&+2\left(g(\alpha^{-},\alpha^{-},\nu_0)+ g(\alpha^{-\prime},\alpha^{-},\nu_0)\right)
-3\times2-4\times2\\
\geq&2(1+1)+2(2+1)+2(1+2)-6-8\\
=&2>0.
\end{align*}

(3) Suppose that $\hat{\lambda}_2=1$. Then $\lambda=(a,2,1^b)$. $\lambda\ominus1$ and $\sigma\ominus1$ contain partitions of Durfee size 1, where $\sigma\in\lambda\ominus1$.

Firstly, assume that $a>2$ and $b>0$. Then $|\lambda\ominus1|=3$ and there exists only one hook partition in $\lambda\ominus1$ with the form $\sigma_0=(a,1^{b+1})$ and $\sigma_0\ominus1=\{(a-1,1^{b+1}), (a,1^b)\}$. By direct computation we have $f(\sigma_0)=1$. All the other two partitions in $\lambda\ominus1$ have Durfee size 2, the estimation of $f(\sigma)$ can be attributed to condition (2). Secondly, assume that $a\geq2$ and $b=0$. Since $k\geq3$, we have $a\geq6$, $|\lambda\ominus1|=2$, $|\sigma\ominus1|=2$ for $\sigma\in\lambda\ominus1$. By (\ref{eq:gmgmlb}) it is not hard to get $g(\gamma,\gamma,\lambda)=5$. For conditions when $a=2$ and $b\geq0$, similar discussions can be used by taking transpose to $\lambda$.
\end{proof}

\begin{proposition}
For $\gamma=(k+1,2,1^{k-1})\vdash 2k+2$, let $s_\gamma*s_\gamma=\sum_{\lambda\in P(2k+2)}g(\gamma,\gamma,\lambda)s_\lambda$. Suppose that  $d(\lambda)=3$ and let $\hat{\lambda}=(\hat{\lambda}_1,\hat{\lambda}_2,\hat{\lambda}_3)$. Then $g(\gamma,\gamma,\lambda)>0$ if and only if $\hat{\lambda}_3=1$ or $2$. Moreover, if $\hat{\lambda}_3\geq3$ or  $d(\lambda)\geq4$ then $g(\gamma,\gamma,\lambda)=0$.
\end{proposition}

\begin{proof}
Suppose $d(\lambda)=3$ where $\hat{\lambda}=(\hat{\lambda}_1,\hat{\lambda}_2,\hat{\lambda}_3)$. If $\hat{\lambda}_3\geq3$, then there exist no partitions with Durfee size less than 3 in $\lambda\ominus1$ and $\sigma\ominus1$ where $\sigma\in \lambda\ominus1$. So we have $g(\gamma,\gamma,\lambda)=0$ by Table \ref{tab:kchk} and (\ref{eq:gmgmlb}). Similar result holds for $\lambda$ when $d(\lambda)\geq4$.

Assume that $\hat{\lambda}_3=1$. Then there exists only one partition in $\lambda\ominus1$ with Durfee size 2 denoted by $\sigma_0$. In (\ref{eq:gmgmlb}),  the only possible negative part is
$$f(\sigma_0)=\sum_{\nu\in\sigma_0\ominus1}2\left(g(\alpha^{-},\alpha^{-},\nu)+ g(\alpha^{-\prime},\alpha^{-},\nu)\right)-3 g(\alpha,\alpha,\sigma_0)-4g(\tilde{\alpha},\alpha,\sigma_0).$$
Since $\sigma_0=(\hat{\lambda}_1,\hat{\lambda}_2)$ and $\hat{\lambda}_2\geq3$,
by the proof of condition (1) in Theorem \ref{thm:srsrds2}, we can obtain $f(\sigma_0)>0$ and so $g(\gamma,\gamma,\lambda)>0$.


Similar arguments can be used for the discussion of $\hat{\lambda}_3=2$.
\end{proof}

The following theorem illustrates the effectiveness of character criterion in \cite[Lem. 1.3]{PPV} under the example we have discussed.
\begin{theorem}
Let $\gamma_1=(k+1,1^k)$, $\gamma_2=(k+1,2,1^{k-1})$ be self-conjugate. Suppose that $A_i=\{\lambda\in P(n)|g(\gamma_i,\gamma_i,\lambda)>0\}$ and $B_i=\{\lambda\in P(n) |\chi^{\lambda}(\hat{\gamma}_i)\neq0\}$ ($i=1,2$). Then $\frac{|B_i|}{|A_i|}\to 0$ as $k\to \infty$.
\end{theorem}
\begin{proof}
Denote $n_1=2k+1$ and $n_2=2k+2$. By Table \ref{tab:kchk} and Theorem \ref{thm:srsrds2}, we know that $A_i$ contains $DS(2,n_i)$. Next, we find a lower bound for $|DS(2,n_1)|$. The estimation of $|DS(2,n_2)|$ is similar. For $\lambda\in DS(2,n_1)$ with $\hat{\lambda}=(\hat{\lambda}_1,\hat{\lambda}_2)$, we have
$\frac{n_1+1}{2}\leq\hat{\lambda}_1\leq n_1-1$. Moreover, if $\hat{\lambda}_1=j$ with $\frac{n_1+1}{2}\leq j\leq n_1-1$, then $\lambda_1$ can be chosen as $2$, $3,\ldots,j-1$. Hence, for each $\lambda_1=j$ there exist at least $j-2$ partitions with Durfee size 2. So we have
$$|DS(2,n_1)|\geq\sum_{(n_1+1)/2\leq j\leq n_1-1}(j-2)=\frac{(3n_1-9)(n_1-1)}{8}.$$

It is well known that $\chi^{\lambda}(\hat{\gamma}_1)\neq0$ if and only if $\lambda\in DS(1, n_1)$ (i.e. $\lambda$ is a hook partition). Hence we have $|B_1|=n_1$ and
 $$\frac{|B_1|}{|A_1|}\leq \frac{|B_1|}{|DS(2,n_1)|}\leq \frac{8n_1}{(3n_1-9)(n_1-1)},$$
so we have  $\frac{|B_1|}{|A_1|}\to 0$ as $k\to \infty$.

Since $\chi^{\lambda}(\hat{\gamma}_2)\neq0$ if and only if $\lambda\in\{(n_2),(1^{n_2}),(a,2,1^{n_2-a-2})|2\leq a\leq n_2-2\}$, we have $|B_2|=n_2-1$. By similar estimation of $|DS(2,n_2)|$, we also have   $\frac{|B_2|}{|A_2|}\to 0$ as $k\to \infty$.
\end{proof}


\bibliographystyle{amsplain}

\end{document}